\documentclass[11pt]{article}
\usepackage[margin=1in]{geometry}
\usepackage[utf8]{inputenc}
\usepackage[T1]{fontenc}
\usepackage{amsmath,amssymb}
\usepackage[bookmarks, colorlinks=true, plainpages = false, citecolor = blue,linkcolor=blue,urlcolor = blue, filecolor = blue]{hyperref}
\usepackage{cleveref}
\usepackage{comment}
\usepackage{amsthm}
\usepackage{dsfont}
\renewcommand{\operatorname}{\mathsf}

\newcommand{\E}{\mathbb{E}}
\renewcommand{\tilde}{\widetilde}
\newcommand{\R}{\mathbb{R}}
\renewcommand{\P}{\mathbb{P}}

\newcommand{\eps}{\varepsilon}
\newcommand{\one}{\mathds{1}}

\usepackage{graphicx,graphics}
\usepackage{enumerate}

\newtheorem{lemma}{Lemma}

\newtheorem{condition}{Condition}
\newcommand{\bmid}{\,\middle\vert}

\newtheorem{theorem}{Theorem}
\newtheorem{proposition}{Proposition}

\newcommand{\I}{\mathbb{I}}
\renewcommand{\d}{\mathrm{d}}
\renewcommand{\det}{\operatorname{det}}
\newcommand{\fr}{\operatorname{Fr}}

\newcommand{\lr}[1]{\left(#1\right)}

\newcommand{\abs}[1]{\left|#1\right|}

\usepackage[framemethod=TikZ]{mdframed}
\crefname{condition}{condition}{conditions}
\Crefname{condition}{Condition}{Conditions}
\newcommand{\gives}{\rightarrow}
\newcommand{\norm}[1]{\left|\left| #1 \right|\right|}
\newcommand{\normm}[1]{\left|\left| #1 \right|\right|_{(k)}}
\usepackage[numbers]{natbib}  
\title{Global Universality of Singular Values in Products\\of Many Large Random Matrices}
\date{}
\author{
Boris Hanin\footnote{
Princeton University. Email: bhanin@princeton.edu
}
,
Tianze Jiang\footnote{
Princeton University. Email: tzjiang@princeton.edu
}
}
\begin{document}
\maketitle
\begin{abstract}
    We study the singular values (and Lyapunov exponents) for products of $N$ independent $n\times n$ random matrices with i.i.d. entries. Such matrix products have been extensively analyzed using free probability, which applies when $n\gives \infty$ at fixed $N$, and the multiplicative ergodic theorem, which holds when $N\gives \infty$ while $n$ remains fixed. The regime when $N,n\gives \infty$ simultaneously is considerably less well understood, and our work is the first to prove universality for the global distribution of singular values in this setting. Our main result gives non-asymptotic upper bounds on the Kolmogorov-Smirnoff distance between the empirical measure of (normalized) squared singular values and the uniform measure on $[0, 1]$ that go to zero when $n, N\gives \infty$ at any relative rate. We assume only that the distribution of matrix entries has zero mean, unit variance, bounded fourth moment, and a bounded density. Our proofs rely on two key ingredients. The first is a novel small-ball estimate on singular vectors of random matrices from which we deduce a non-asymptotic variant of the multiplicative ergodic theorem that holds for growing matrix size $n$. The second is a martingale concentration argument, which shows that while Lyapunov exponents at large $N$ are not universal at fixed matrix size, their empirical distribution becomes universal as soon as the matrix size grows with $N$. 
   
\end{abstract}
\section{Introduction}
This article concerns the distribution of singular values for products of independent random matrices 
\begin{equation}\label{eqn:prod_main_def}
X_{N, n} \triangleq W_N \cdots W_1,\qquad W_i\in\R^{n\times n},
\end{equation}
with entries of $\sqrt{n} W_i$ drawn i.i.d. from a fixed  distribution $\mu$. We assume $\mu$ satisfies the following
\begin{condition}\label{cond:main_mu}
The probability measure $\mu$ has zero mean, unit variance, a finite fourth moment $M_4 <\infty$, and a density with respect to the Lebesgue measure bounded above by $K_\infty<\infty$. 
\end{condition}
\noindent Our main result, \Cref{thm:sing_val_approx_unif}, is a quantitative universality result for 
\[
\rho_{N,n} \triangleq \frac{1}{n}\sum_{i=1}^n \delta_{s_i(X_{n,N})^{2/N}},
\]
the empirical distribution of rescaled singular values $s_{1}(X_{N,n})\geq \cdots \geq s_n(X_{N,n})$ of $X_{N,n}$. 
\begin{theorem}\label{thm:sing_val_approx_unif}
    Under \Cref{cond:main_mu}, there exist constants $c_1, c_2, c_3, c_4>0$ depending on $K_\infty, M_4$ with the following property. For all $\varepsilon \in\left(0, 1/2\right)$, if $N>c_1 \eps^{-2}$ and $n>c_2 \eps^{-2} \cdot \log(1/\eps)$, then 
\begin{equation}\label{eqn: thm 1 main}
\mathbb{P}\left(d_{KS}(\rho_{N,n},\, \mathrm{U}_{[0,1]})>\eps\right)\leq  c_3\exp\{-c_4nN\eps^2/\log n\},
\end{equation}
where $d_{KS}$ is the Kolmogorov-Smirnoff distance and $\mathrm{U}_{[0,1]}$ is the uniform distribution on $[0,1]$.
\end{theorem}

\Cref{thm:sing_val_approx_unif} is the first universality result for $\rho_{N,n}$ which holds for a general class of distributions $\mu$ \emph{regardless} of the relative size of $n,N$. It guarantees that as soon as $\mu$ matches the first two moments of a standard Gaussian (and has bounded fourth moment and $L_\infty$ norm), the empirical singular value distribution is close to the case when $\mu$ is a (real) standard Gaussian (see Theorem 1.2 in \cite{hanin2021non}). Our result also extends to complex distributions with independent real and imaginary parts (see \Cref{sec: limitations}). The nature of the universality underlying \Cref{thm:sing_val_approx_unif}, however, is unusual in the following two senses:
\begin{itemize}
    \item \textbf{Requirement of bounded density.} Universality is a hallmark of random matrix theory in the regime where the matrix size $n$ tends to infinity. It is uncommon in such universality results to include the hypothesis that $\mu$ has a bounded density. But such an assumption is essential in our setting because we are interested in the setting of growing $N$. For example, the matrix product $X_{N,n}$ has rank one with positive probability if $\mu$ contains an atom and $N$ is exponential in $n^2$. In the ergodic limit of fixed $n$ and diverging $N$, moreover, the limiting empirical measure $\rho_{n,\infty}$ of singular values is known to be \textit{non-universal} and depends on $\mu$, see e.g. \cite{kargin2014largest, akemann2014universal,avila2023continuity}.  This is in contrast the free probability limit where it is typical to consider polynomials of fixed degree (e.g. fixed $N$) evaluated at a collection of $n\times n$ random matrices with $n\gives \infty$ (see e.g. \cite{gotze2010asymptotic, o2010products, gotze2021rate}).
    \item \textbf{One global shape, many local shapes.} For many classical random matrix ensembles,  universality holds not just for the global distribution of eigenvalues or singular values but persists also at the microscopic scale where consecutive eigenvalues or singular values remains order $1$ apart as $n\gives \infty$. In our setting, however, even in the simplest case when $\mu$ is a standard complex Gaussian, the \textit{local} distribution depends on the limiting value of $n/N$ \cite{akemann2014universal,liu2018lyapunov}. The relative size of $n,N$ therefore determines the local statistics but does not impact the global properties of $\rho_{N,n}$. It remains open both to determine at what scale the local distribution of singular values begins to depend on $N/n$ and whether the local limits, derived using methods from integrable systems, are universal (see \cite{hanin2020products} for some partial progress). 
\end{itemize}

While the effect of simultaneously large $n,N$ on the statistics $\rho_{n,N}$ is still far from understood, prior work showed that the \textit{global} distribution of singular values $\rho_{n,N}$ converges to $\mathrm{U}_{[0,1]}$ if one either first takes $n\gives\infty$ and then $N\gives \infty$ or vice versa \cite{isopi1992triangle, kargin2008lyapunov}. These articles use very different tools: the work \cite{kargin2008lyapunov}, which first takes $n\gives \infty$, relies on free probability while \cite{isopi1992triangle} uses the multiplicative ergodic theorem to analyze what happens when one first takes $N\gives \infty$.

Neither free probability nor ergodic techniques are simple to make effective when both $n$ and $N$ are large but finite. To make progress in this direction the article \cite{hanin2021non} used small ball probabilities to quantify, at finite $N$, the rate of the convergence in the multiplicative ergodic theorem and obtain a sharper version of \Cref{thm:sing_val_approx_unif} in the special case when $\mu$ is the standard (real) Gaussian (see \Cref{sec: limitations} for a discussion of our optimality). We take a similar approach. The core difference is that the distribution of the individual $W_i$ matrices is no longer isotropic (invariant under left or right rotations). As we explain in \Cref{sec: proof short}, this means we must obtain new small ball probabilities for the inner product between a fixed $k$-frame in $\R^n$ and the projection onto the span of the top $k$ singular vectors for $X_{N,n}$.

\subsubsection*{Outline of Remainder of the Article.} The rest of this article is organized as follows. First, in \Cref{sec: related works}, we give a more thorough review of the relation between our results and prior work. Then in \Cref{sec: proof short} we state the main results needed to prove \Cref{thm:sing_val_approx_unif}. We make some further remarks on the results as well as future directions in \Cref{sec: limitations}. The remaining proofs of these results are provided in \Cref{sec: remainder}, after a brief recall of auxiliary technical results needed in \Cref{sec: preliminaries}.

\section{Related works}\label{sec: related works}
Products of random matrices are a vast subject. We provide here some representative references, focusing mainly on work in which the number of matrices, $N$, is large or growing.

The setting where the matrix size $n$ is fixed while the the number of terms in the matrix product $N$ grows has attracted much interest starting from the seminal work of Furstenberg \cite{furstenberg1960products} and later of Oseledec \cite{oseledec1968multiplicative} on the multiplicative ergodic theorem. Particularly relevant to the present article are the works of Newman \cite{newman1986distribution} and Isopsi-Newman \cite{isopi1992triangle}. Since then, the study of Lyapunov exponents of random matrix products has found applications to the study of random Sch\"odinger operators \cite{bougerol2012products}, number theory and dynamics  \cite{kontsevich1997lyapunov, masur2002rational}, and beyond \cite{wilkinson2017lyapunov, arnold1995random}.

Matrix products when $n\gives \infty$ while $N$ is potentially large but fixed have also been extensively studied. For instance, classical results in free probability concern the spectrum of products of a fixed number of (freely) independent matrices \cite{mingo2017free, voiculescu1992free, nica2006lectures}. In this vein, the articles \cite{tucci2010limits, kargin2008lyapunov} both use tools from free probability to obtain the analog of \Cref{thm:sing_val_approx_unif} in the setting where first $n\gives \infty$ and then $N\gives \infty$. Prior work has also taken up a non-asymptotic analysis of eigenvalues \cite{gotze2010asymptotic,gotze2021rate} for such matrix products as well as the local distribution of their singular values \cite{liu2016bulk}. 

The setting when $n,N$ simultaneously grow is less well understood but has nonetheless attracted significant interest in recent years. For example, we point the reader to a beautiful set of the articles that use techniques from integrable systems and integrable probability to study singular values for products of iid complex Ginibre matrices and related integrable matrix ensembles. These include the works \cite{akemann2014universal,akemann2012universal,akemann2019integrable,akemann2019universal,burda2012spectral,burda2013commutative} which, at a physics level of rigor, were the first to analyze the asymptotic distribution of singular values for products of iid complex Ginibre matrices. Some of the results in the preceding articles were proved rigorously in \cite{liu2018lyapunov}. We also point the interested reader to \cite{ahn2022fluctuations, gorin2018gaussian} another perspective on how to use techniques from integrable probability to study such matrix products. The study of the singular values of $X_{N,n}$ when $n,N$ are both large has also received attention due to its connection with the spectrum of input-output Jacobians in randomly initialized neural networks \cite{pennington2017resurrecting, pennington2018emergence, DBLP:conf/aistats/PenningtonSG18, hanin2020products, hanin2018neural}.

\section{Main ideas and proof outline}\label{sec: proof short}
As we will explain in this section, there are three key steps in the proof of \Cref{thm:sing_val_approx_unif}. To present them, let us agree on some notations. We write $\fr_{n,k}=\{X\in\R^{n\times k}: X^TX = \I_k\}$ for the space of $k$-frames in $\R^n$ (i.e. orthonormal systems of $k$ vectors). For any matrix $A\in\R^{a\times b}, a, b\geq k$, we write:
$\|A\|_{(k)} \triangleq \prod_{i=1}^k s_i(A)$
the product of top $k$ singular values.
For any $n\times k$ matrix $X$ we thus have
\begin{equation}\label{eq:wedge-norm-def}
\normm{X} =\prod_{i=1}^ks_i(X) =\sqrt{\det(X^TX)}    
\end{equation}
following the Gram identity. Unless specified otherwise, all constants are finite, positive, and may depend on $M_4, K_\infty$, the moments in \Cref{cond:main_mu}.  

\paragraph{Step 1: From many singular values to the top singular value.} To study the singular values of $X_{N,n}$, it will be convenient to study their partial products:
\begin{equation}\label{eq:XV-def}
\norm{X_{N,n}}_{(k)} = \prod_{i=1}^k s_i(X_{N,n}) =  \sup_{U\in \fr_{n,k}}\normm{X_{N,n}U},     \qquad k = 1,\ldots, n,
\end{equation}
where the first equality is a definition and the second equality follows from standard linear algebra. The representation on the right of \eqref{eq:XV-def} recasts the product of the top $k$ singular values of $X_{N,n}$ as the top singular value for the action of $X_{N,n}$ on the space of $k$-frames in $\R^n$. This is useful since analyzing the top singular value, or equivalently top Lyapunov exponent, is a natural and well-studied way to understand the long time behavior of a dynamical system. This is precisely the philosophy of most prior work in the regime where $N\gives \infty$ (see e.g. \cite{isopi1992triangle, furstenberg1960products, hanin2021non, le2006theoremes}).

\paragraph{Step 2: Removing the supremum in \eqref{eq:XV-def}.} One advantage of the representation on the right of \eqref{eq:XV-def} is that each term $\normm{X_{N,n}U}$ inside the supremum can naturally be thought of as simple sub-multiplicative functional of the state of a random dynamical system after $N$ steps. In this analogy, the frame $U$ determines the initial condition and the evolution from time $t-1$ to time $t$ consists of multiplying by $W_t$. 

The second and most important technical step in our proof of \Cref{thm:sing_val_approx_unif} is to show that once $N$ is large, we can approximately drop the supremum over frames in \eqref{eq:XV-def}. That this is possible is a key conceptual insight that goes back to \cite{furstenberg1960products}, which showed that for a wide range of entry distributions $\mu$, we have
$$\lim _{N \rightarrow \infty}\frac{1}{N} \log \lr{\frac{\sup_{V\in \fr_{n,k} }\normm{X_{N,n}V}}{\normm{X_{N,n}U_0}}}=0$$
on fixed $U_0\in \fr_{n,k}$. The previous displayed equation is a consequence of the multiplicative ergodic theorem, which guarantees that as $N\gives \infty$ the supremum, the average, and the pointwise behavior of $\normm{X_{N,n}V}$ is the same for almost every frame $V$. Since we seek to describe the distribution of singular values of $X_{N,n}$ when $N$ is finite, we will need a quantitative version of this result. This is the content of \Cref{prop:sup_eq_any}, which is more conveniently phrased in terms of the Lyapunov exponents of $X_{N,n}$:
\[
\lambda_i = \lambda_i(X_{N,n})=\frac{1}{N}\log s_i(X_{N,n})=i\text{-th Lyapunov exponent},\qquad \sum_{i=1}^k \lambda_i = \frac{1}{N}\log \norm{X_{N,n}}_{(k)}
\]
\begin{proposition}[Reduction from sup norm to pointwise norm]\label{prop:sup_eq_any}
     Denote by $U_0=\I_{[k]}^T\in\fr_{n, k}$ the $k$-frame whose columns are the first $k$ standard basis vectors of $\R^n$. Then, assuming \Cref{cond:main_mu} holds, there exist constants $c_1, c_2, c_3>0$ depending only on $K_\infty, M_4$, such that for any $n\geq k$:
    \begin{equation}\label{eqn: small_ball_lya}
    \mathbb{P}\left(\frac{1}{n}\abs{\sum_{i=1}^k \lambda_i-\frac{1}{N}\log \normm{X_{N, n}U_0} }>s\right) \leq c_1 \exp \{-c_2 n N s\}.
    \end{equation}
    for all $ s>c_3\frac{k\log(en/k)}{nN}$. 
\end{proposition}
We prove \Cref{prop:sup_eq_any} in \Cref{sec: reduction proof} and emphasize here only the main ideas. The key observation is that the estimate \eqref{eqn: small_ball_lya} follows from proving that the subspace spanned by the top $k$ singular vectors of $X_{N, n}$ is ``well-spread'' on the Grassmannian with high probability when $N\gg 1$. To explain this, let us consider the simple but illustrative case of $k=1$. Our goal is then to obtain a lower bound for
\[
 0\geq \frac{1}{Nn}\log \|{X_{N,n}v}\|_{(1)}-\frac{1}{n}\lambda_1=\frac{1}{Nn}\log\frac{\|X_{N,n}v\|}{\|X_{N, n}\|_{op}},
\]
where $v=[1,0,0,\dots]^T$. If $s_i$ are the singular values of $X_{N,n}$ and $e_i, f_i$ are the corresponding right and left singular vectors, then 
\[
X_{N,n}v = \sum_{i=1}^n s_i\langle v, e_i \rangle f_j
\]
and we have
\begin{equation}\label{eq:k1-heuristic}
\frac{2}{Nn}\log\frac{\|X_{N,n}v\|}{\|X_{N, n}\|_{op}}= \frac{1}{Nn}\log\frac{\sum_{i=1}^n s_i^2|\langle v, e_i \rangle|^2}{s_1^2}\geq \frac{1}{Nn}\log|\langle v, e_1\rangle|^2.    
\end{equation}
Obtaining lower bounds on $\langle v, e_1\rangle$  is the same as obtaining small ball probabilities for $e_1$ around the orthogonal complement to $v$. Repeating this argument for general $k$ shows that \Cref{prop:sup_eq_any} will follow from the statement that ``the distribution of the top singular vectors does not concentrate on a fixed co-dimension $k$ subspace''. See \Cref{lem:sq_det_small} for the key result to this end.

Prior work \cite{burda2012spectral, burda2013commutative, akemann2014universal, ahn2022fluctuations, gorin2022gaussian, hanin2021non} assumed that the matrices $W_i$ are rotationally invariant and hence that the law of right singular vectors $(e_i)_{i=1}^k$ is Haar for every $k$. In particular, when $k=1$ this implies 
$$\mathbb P\lr{|\langle v, e_1\rangle|\leq \epsilon /\sqrt{n}} \leq C\epsilon$$
for a universal constant $C$. The main technical difficulty in our present setting is that we do not know how to characterize the (joint) distribution of the top $k$ singular vectors of $X_{N, n}$, even when $k=1$. Nonetheless, we obtain small ball probabilities for $\log \lr{\normm{X_{N,n}U_0}/\norm{X_{N,n}}_{(k)}}$ in \Cref{sec: reduction proof} relying only on small ball probabilities for the entrywise measure $\mu$. 

Finally, we mention that the frame $U_0$ from \Cref{prop:sup_eq_any} does not have to be the canonical frame. A different proof shows, that a slightly weaker version of \eqref{eqn: small_ball_lya} holds for \emph{any} $U_0\in\fr_{n, k}$ (see \cref{eqn: small_ball_lya_any}). This yields seemingly new information on the problem of studying singular vectors of $W_1$, which we believe is of independent interest. We defer this discussion to \Cref{sec: any U0}.

\paragraph{Step 3: Doob decomposition and concentration for $\log\norm{X_{N,n}U_0}$ in \eqref{eq:XV-def}.}  In light of \Cref{prop:sup_eq_any}, estimating the partial products $\norm{X_{N,n}}_{(k)}$ comes down to bounding the ``point-wise'' norms $\normm{X_{N, n}U_0}$ for a fixed frame $U_0 \in\fr_{n, k}$, which is done in the following
\begin{proposition}\label{prop: ptwise any}
    Assuming \Cref{cond:main_mu}, there exist constants $c_1, c_2, c_3>0$ depending on $M_4, K_\infty$ such that for any $n, k, N$, any $U_0\in\fr_{n, k}$, and for all $s\geq c_3 n^{-1}\log (ek)$:
    \begin{equation}\label{eqn: concentrate single matt}
    \P\left(\frac{1}{n}\left|\frac{1}{N}\log\normm{X_{N, n}U_0}-\frac{1}{2}\sum_{j=1}^k\log\frac{n-i+1}{n}\right|>s\right)\leq c_1\exp\{-c_2sn/\log (ek)\}.
\end{equation}
\end{proposition}
We prove \Cref{prop: ptwise any} in \Cref{sec: ptwise proof}. The main idea is to express $\log\normm{X_{N, n}U_0}$ as an average. For this, let $U_1 = U_0$ and define $U_{t+1}$ inductively through the singular value decomposition of $W_tU_t$: 
\[
W_tU_t = U_{t+1}\operatorname{diag}(\{s_i(W_tU_t)\}_{i=1}^k)O_{t},\qquad O_{t}\in\fr_{k, k},\, U_{t+1}\in\fr_{n, k}
\]
Then, recalling \eqref{eq:wedge-norm-def} and noting that $U_t$ is a measurable function of $W_1,\ldots, W_{t-1}$, a simple computation gives the following equality in distribution
\begin{equation}\label{eqn: point norm prod decomp}
\frac{1}{N}\log \|X_{N, n}U\|_{(k)}^2 = \frac{1}{N}\sum_{i=1}^N \log \|W_iU_i\|_{(k)}^2.
\end{equation}
A direct computation now shows that given any $U\in \fr_{n,k}$ we have
\begin{equation}\label{eqn: exp norm}
    \E_{\sqrt{n}W\sim\mu^{\otimes n\times n}}\left[\normm{WU}^2\right] =\frac{n!}{n^k(n-k)!} = \prod_{j=1}^k\frac{n-j+1}{n}.
\end{equation}
These expectations determine the constants around which $\frac{1}{N}\log \normm{X_{N,n}U_0}$ concentrates in \Cref{prop:sup_eq_any}. The result then follows from an Azuma-type concentration inequality for random variables with sub-exponential tails (see \Cref{lem: mtg concentrate}).

\paragraph{Combining everything together.}
Putting \Cref{prop:sup_eq_any} and \Cref{prop: ptwise any} we get for $s \gtrsim \log(ek)/n$ that
\begin{equation}\label{thm: main technical concentration}
    \mathbb{P}\left(\frac{1}{n}\left| \sum_{i=1}^k\lambda_i(X_{N, n})-\frac{1}{2} \sum_{j=1}^k \log \frac{n-i+1}{n}\right|>s\right) \leq c_1 \exp \left\{-c_2 s n / \log (e k)\right\} .
\end{equation}
Together with some elementary algebra in \Cref{sec: main proofs end}, we are able to control the cumulative distribution function for the empirical measure $\rho_{N,n}$ of singular values with high probability.

\section{Discussions}\label{sec: limitations}
In this section, we discuss some extensions and limitations of the present work. 
\paragraph{Dependence on $n, N$.} We begin by briefly remarking on the dependence in \Cref{thm:sing_val_approx_unif} and relation \eqref{thm: main technical concentration} on $n, N$. In particular, for fixed $N$, consider an i.i.d. sequence of $\{X_{N, n}\}$ and let $\eps\asymp N^{-1/2}$ one has:
$$
\P\left(\sup _{t \in \mathbb{R}}\left|\frac{1}{n} \#\left\{1 \leq i \leq n \mid s_i^{2 / N}\left(X_{N, n}\right) \leq t\right\}-\mathcal{U}(t)\right| \geq c_1N^{-1/2}\right)\leq 2\exp\{-c_2 n/\log n\}
$$
where $\mathcal{U}(t) \triangleq \P_{x\sim\mathrm{U}[0, 1]}(x\leq t)$. By Borel-Cantelli lemma, this implies that with probability one,
$$
\sup _{t \in \mathbb{R}}\left|\lim _{n \rightarrow \infty} \frac{1}{n} \#\left\{1 \leq i \leq n \mid s_i^{2 / N}\left(X_{N, n}\right) \leq t\right\}-\mathcal{U}(t)\right| \leq \frac{C}{\sqrt{N}}.
$$
The $O(1/\sqrt{N})$ rate can be seen as a Berry-Esseen type bound, see also \cite[Section 1.2]{hanin2021non}. This suggests that our dependence on $N$ is at least comparable to standard CLT rates. However, since we require in \eqref{thm: main technical concentration} that $s\in\Omega(\log (ek)/n)$ (even when $N\to\infty$), is unclear whether the dependence on $n$ is optimal. This dependence, unfortunately, cannot be improved significantly based on current techniques. To illustrate this, consider $k=1$. The dependence of the mean $\E[\log\|WU_0\|^2]=\E_{x_i\sim_{i.i.d.}\mu}[\log\sum_{i=1}^n x_i^2]$ on $\mu$ can be shown to be $\Omega(1/n)$ (i.e. there exists different $\mu$'s such that the expectation differs by $\Omega(1/n)$). The studies of more fine-grained behaviors of Lyapunov exponents in the ergodic regime (when universality does not hold for $n<\infty$) are left open to future works.

\paragraph{Extension to complex matrices.} 
While our proofs are formulated for real matrices, we remark that all results can be directly extended to complex random variables under the following assumptions on $\mu$:
\begin{condition}\label{cond:complex entries}
    Suppose that each entry of the $\sqrt{n}W_i$'s are i.i.d. drawn from a distribution $X+Yi$, where $X\sim\mu_X$  and $Y\sim\mu_Y$ are independent (real) random variables satisfying (1) zero mean and unit variance: $\E[X]=\E[Y]=0, \E[X^2+Y^2]=1 $; (2) finite fourth moment $\E[X^4+Y^4]\leq M_4 <\infty$; and (3) density bounded above $\|\mu_X\|_{L_\infty}, \|\mu_Y\|_{L_\infty}\leq K_\infty<\infty$. 
\end{condition}
Extending both \Cref{prop:sup_eq_any} and \Cref{prop: ptwise any} mainly requires replacing transpose with conjugate transpose and absolute values with norms does not require much conversions at all. The only nontrivial technical difference lies in transforming existing small ball estimates for real random variables (which appears in \Cref{lem:small_ball_proj_applied}) to their complex analog. This can be done by noting that a complex frame $A+Bi\in\fr_{n, k}$ can be converted into a real frame $\begin{bmatrix}
    A & -B\\
    B & A
\end{bmatrix}\in\fr_{2n, 2k}$ (and defining the {density} of a $\mathbb{C}^d$ complex distribution to be the density of the distribution of their canonical real decomposition in $\R^{2d}$). As a result, \Cref{thm:sing_val_approx_unif} and \eqref{thm: main technical concentration} holds under \Cref{cond:complex entries} as well with different set of constants (versus \Cref{cond:main_mu}). Moreover, in fact,
our proof only needs that $W_i$'s are independent and distributed according to $W_i=_dA_iP_i$ where $P_i$ is any fixed orthogonal matrix and $A_i\sim\mu_i^{\otimes n\times n}$ for some $\mu_i$ satisfying prescribed conditions.

\paragraph{Dependence on constants $M_4, K_\infty$.}
A careful analysis of the proof shows that the constants $c_1,c_2,c_3,c_4$ appearing in the statement of \Cref{thm:sing_val_approx_unif} can be taken to depend on constants $M_4, K_\infty$ as follows
\begin{equation}\label{eqn: explicit constants}
    c_1 \in O(K_\infty^\delta\cdot M_4),\quad c_2 \in O(K_\infty^\delta\cdot M_4),\quad c_3=2,\quad c_4^{-1}\in O(K_\infty^\delta\cdot M_4),
\end{equation}
where $\delta>0$ is arbitrary but fixed and the implicit constants in the big-O terms are universal. 

\section{Review of auxiliary technical results}\label{sec: preliminaries}
Before we complete all the deferred proofs, we collect below we several technical results use in our main proofs and establish some notation. 

\paragraph{Notation.} We use $=_d$ to denote equivalence in distribution. Unless specified otherwise, we use $\land$ to denote the minimum and $\lor$ the maximum of two numbers. We denote $[r]\triangleq\{1,2,\dots, r\}, r\in\mathbb{N},$ and ${\binom{[r]}{t}}\triangleq\{I\subseteq[r]: |I|=t\}$.
When not specified otherwise, we write (for a $m\times n$ matrix $A$) $A_i\in\R^n$ as the $i$-th row of $A$ for $i\in [m]$ and $A_I\in\R^{|I|\times n}$ as the $|I|\times n$ submatrix for $I\subseteq [m]$. Furthermore, we denote $s_1(A)\geq s_2(A)\geq\dots$ be the ordered singular values of any matrix $A$.

\paragraph{An isotropic inequality for right product with random uniform frames.}
We will examine the effect of a ``uniformly random'' frame being applied on any matrix. 
\begin{lemma}[See also Section 9 in \cite{hanin2021non}]\label{lem: random wedge product}
    There exists a constant $c$ with the following property. Suppose $G$ is sampled from the Haar measure on $\fr_{n,k}$. For any invertible matrix $M\in\R^{n\times n}$
    $$\mathbb{P}\left(\left(\frac{\|MG\|_{(k)}}{\|M\|_{(k)}}\right)^{\frac{1}{k}} \leq \varepsilon \sqrt{\frac{k}{n}}\right) \leq(c \varepsilon)^{\frac{k}{2}}.$$
\end{lemma}
\begin{proof}
    Note that there exists $L, R\in\fr_{n, k}$ such that:
    $$L^TM = \operatorname{diag}(\{s_i(M)\}_{i=1}^k)R^T$$
    so
    $$\|MG\|_{(k)}\geq|\det(L^TMG)| = \|M\|_{(k)}|\det(R^TG)|$$
    and hence 
    $$
    \frac{\|MG\|_{(k)}}{\|M\|_{(k)}}\geq |\det(R^TG)|
    $$
    where $R$ is a fixed frame. 
    The rest follows from Corollary 9.4 and equation (9.1) in \cite{hanin2021non}.
\end{proof}

\paragraph{Sub-multiplicativity for products of top singular values.}
We will use the following result, which shows that the $\|\cdot \|_{(k)}$ norm is sub-multiplicative.
\begin{lemma}[See also \cite{gelfand1950unitary}]\label{lem: sub mult k norm}
    For any two matrices $A, B\in\R^{n\times n}$, one has:
    \begin{equation}\label{eqn: sub prod k norm}
    \|AB\|_{(k)}\leq \|A\|_{(k)}\|B\|_{(k)}
\end{equation}
\end{lemma}
\begin{proof}
For any two matrices $A, B\in\R^{n\times n}$ and frames $L,  R\in\fr_{n, k}$ such that $\det(L^TABR) = \|AB\|_{(k)}$, the standard SVD gives: $L^TA = \Sigma^{(L)}_{k\times k}R_1^T$ and $BR = L_1\Sigma^{(R)}_{k\times k}$ where $L_1, R_1\in\fr_{n, k}$ so:
$$\|AB\|_{(k)}=\det(L^TABR)=\prod_{i}(\Sigma^{(L)}_{k\times k})_i(\Sigma^{(R)}_{k\times k})_i\cdot \det(R_1^TL_1)\leq \prod_{i}(\Sigma^{(L)}_{k\times k})_i(\Sigma^{(R)}_{k\times k})_i\leq \|A\|_{(k)}\|B\|_{(k)}$$which concludes our claim.    
\end{proof}
\paragraph{Useful results on small-ball probability.}
To control the small-ball density of projections, we will use the following result. In fact, this is the \emph{only} place in which we needed bounded Lebesgue density in \Cref{cond:main_mu}.
\begin{lemma}[Theorem 1.1 of \cite{rudelson2015small}]\label{lem:small_ball_proj_rv}
   Let $X=\left(X_1, \ldots, X_n\right)$ where $X_i$ are real-valued independent random variables. Assume that the densities of $X_i$ are bounded by $K$ almost everywhere. Let $P\in\fr_{n, d}$ be an orthogonal projection from $\mathbb{R}^n$ onto $\R^d$. Then the density of the random vector $P^TX$ is bounded by $(C K)^d$ almost everywhere, where $C$ is a positive absolute constant. Furthermore, when $d=1$ and $P$ is a vector with norm 1, the max density of $\langle P,X \rangle$ is at most $\sqrt{2}K$.
\end{lemma}
As a corollary, we can show that
\begin{lemma}\label{lem:small_ball_proj_applied}
     Let $X=\left(X_1, \ldots, X_n\right)$ where $X_i$ are real-valued independent random variables. Assume that the densities of $X_i$ are bounded by $K$ almost everywhere. Let $P\in\fr_{n, d}$ be a orthogonal projection in $\mathbb{R}^n$ onto $\R^d$. Then
     $$\P\left(\|PX\|_2^2\leq d\cdot s\right)\leq (C_1K\sqrt{s})^d$$
     where $C_1$ is an absolute constant.
\end{lemma}
\begin{proof}
Consider the Lebesgue measure on a $d$-dimensional ball with radius $\sqrt{ds}$. Its volume $V_d(\sqrt{ds})$ satisfies:
$$V_d(\sqrt{d s})\leq \left(C s^{1 / 2}\right)^d$$
where $C$ is a universal constant independent of $d$. Hence, via \Cref{lem:small_ball_proj_rv} we get:
$$\P\left(\|P^TX\|_2^2\leq d\cdot s\right)\leq\left(C_d s^{1 / 2}\right)^d\cdot (CK)^d \leq (\tilde C K\sqrt{s})^d$$
for some universal $\tilde C$.
\end{proof}
\paragraph{Useful results on sub-exponential random variables.} We collect here some simple results on sub-exponential random variables. We begin with an elementary result:
\begin{lemma}\label{lem: exp_mean_upp}
    If a random variable $X$ with constants $c_1>1$ and $c_2>0$ such that:
    $$\P(|X|\geq t)\leq c_1\exp(-t/c_2)$$
    for all $t>0$
    then $\P(|X|\geq t)\leq 2\exp[-t/(2c_1c_2)]$ for all $t>0$.
\end{lemma}
\begin{proof}
Since probabilities are always bounded above by 1, 
this can be easily verified via checking:
    $$
    c_1\exp(-t/c_2) \land 1\leq 2\exp[-t/(2c_1c_2)]\land 1
    $$
    for all $c_1>1, c_2>0, t\geq 0$.
\end{proof}

We now recall the usual equivalent definitions of a sub-exponential random variable. 

\begin{lemma}[Sub-exponential properties,  Proposition 2.7.1 in \cite{VershyninHDPBook}]\label{lem: HDP SubE}
Let $X$ be a random variable. Then the following properties are equivalent: the parameters $K_i>0$ appearing in these properties differ from each other by at most an absolute constant factor.
\begin{enumerate}
\item The tails of $X$ satisfy
$$
\mathbb{P}\{|X| \geq t\} \leq 2 \exp \left(-t / K_1\right) \quad \text { for all } t \geq 0
$$
\item The moments of $X$ satisfy
$$
\|X\|_{L^p}=\left(\mathbb{E}|X|^p\right)^{1 / p} \leq K_2 p \quad \text { for all } p \geq 1
$$
\item 
    The MGF of $|X|$ satisfies
$$
\mathbb{E} \exp (\lambda|X|) \leq \exp \left(K_3 \lambda\right) \quad \text { for all } \lambda \text { such that } 0 \leq \lambda \leq \frac{1}{K_3}
$$
\item The MGF of $|X|$ is bounded at some point, namely
$$
\mathbb{E} \exp \left(|X| / K_4\right) \leq 2
$$
\end{enumerate}
Moreover, if $\mathbb{E} [X]=0$ then previous properties are also equivalent to the following one: the MGF of $X$ satisfies
$$\mathbb{E} \exp (\lambda X) \leq \exp \left(K_5^2 \lambda^2\right)$$ for all $\lambda$ such that $|\lambda| \leq \frac{1}{K_5}$.
\end{lemma}
\begin{lemma}[Sub-exp properties for almost centered random variables]\label{lem: Sub-exp properties for almost centered random variables}
    There exists a universal $c_0>0$, such that the following holds. 
    If a random variable $Z$ satisfies for some constants $c_1$:
    $$\P(|Z|\geq t)\leq 2\exp(-t/c_1),$$
    then
    $$
\mathbb{E}\left[e^{\lambda Z} \right] \leq \exp \left\{2\left(c_0 c_1\right)^2 \lambda^2\right\}, \quad\quad |\lambda| \in\left[\frac{2}{c_0^2c_1}, \frac{1}{c_0c_1}\right].
$$
\end{lemma}
\begin{proof}
First of all, denote $\mu = \E[Z]$ then
$$
|\mu|\leq \E[|Z|] = \int_{0}^{\infty}\P(|Z|\geq t)\d t\leq 
\int_{0}^{\infty}2\exp(-t/c_1)\d t = 2c_1.
$$
By \Cref{lem: HDP SubE}, for all $p=1,2,\dots,$
$$
\left(\E[|Z|^p]\right)^{1/p}\leq c_0c_1p
$$
for some fixed constant $c_0$. Thus for $p=1,2,\dots,$
$$
\left(\E[|Z-\mu|^p]\right)^{1/p}\leq \left(\E[(2|Z|\lor2|\mu|)^p]\right)^{1/p}\leq 2(c_0c_1p \lor |\mu|)\leq (2c_0\lor 2) c_1p
$$
and hence (again via \Cref{lem: HDP SubE}) for all $|\lambda|<(\tilde c_0c_1c_2)^{-1}$ one has:
$$
\E[e^{\lambda (Z-\mu)}]\leq \exp\{(\tilde c_0c_1)^2\lambda^2\}.
$$
or
$$
\E[e^{\lambda Z}]\leq \exp\{\lambda\mu +(\tilde c_0c_1)^2\lambda^2\}.
$$
Thus, so long as $|\lambda\mu|\leq (\tilde c_0c_1)^2\lambda^2$, or $|\lambda|\geq |\mu|/(\tilde c_0c_1)^2$, the mean deviation is dominated and we get the desired results for some universal constant $\tilde c_0$.
\end{proof}
\paragraph{Moment inequalities for almost-martingale stochastic processes}
\begin{lemma}[See also \cite{shamir2011variant}]
\label{lem: law of total exp in stoch proceses}
Suppose there is a stochastic process $\{Z_i\}_{i=1}^T$ along with a filtration $\{\mathcal{F}_i\}_{i=1}^T$ and that for all $i$ and some function $\{G_i(\cdot)\}$ one has:
$$
\E[G_{i+1}(Z_{i+1})|\mathcal{F}_i\}]\leq a_i \quad\quad\text{almost surely}
$$
for a sequence of positive numbers $\{a_i\}$, then for any fixed $T$:
$$\E\left[\prod_{i=1}^T G_i(Z_i)
    \right]\leq \prod_{i=1}^{T}a_i.$$
\end{lemma}
\begin{proof}
We use the law of total expectations multiple times
    \begin{align*}
    \E\left[\prod_{i=1}^T G_i(Z_i)
    \right]
    &= \E\left[\E\left[\prod_{i=1}^T G_i(Z_i)
    \bmid
    \mathcal{F}_{T}\}\right]\right]
    \\
    & = \E\left[\E\left[G(Z_T)
    \bmid
    \mathcal{F}_{T-1}\}\right]\E\left[\prod_{i=1}^{T-1} G_i(Z_i)
    \bmid
    \mathcal{F}_{T-1}\}\right]\right]
    \\&\leq a_T\cdot \E\left[\E\left[\prod_{i=1}^{T-1} G_i(Z_i)
    \bmid
    \mathcal{F}_{T-1}\}\right]\right]
    \\
    \dots\quad &\leq\prod_{i=1}^{T}a_i.
\end{align*}
\end{proof}
\begin{lemma}\label{lem: mtg concentrate}
    There exists a universal $c_0$ such that the following holds. For stochastic process $\{Z_i\}$ on a filtration $\{\mathcal{F}_i\}$ where for all $i$, with probability one over $\mathcal{F}_{i-1}$, the next item satisfies:
    $$\P\left(|Z_i|\geq t \mid \mathcal{F}_{i-1}\right)\leq c_1\exp(-t/c_2)$$ for some  $c_1>1, c_2>0$. Then,
    for any  $\epsilon\geq 2(c_0c_1c_2)$, 
    \begin{equation}\label{eqn: mtg large off}
        \P\left(\left|\frac{1}{T} \sum_{t=1}^T Z_t\right|>\epsilon\right)\leq2\exp\{T\cdot (2-(c_0c_1c_2)^{-1}\epsilon)\}.
    \end{equation}
\end{lemma}
\begin{proof}
We call \Cref{lem: exp_mean_upp} as well as \Cref{lem: Sub-exp properties for almost centered random variables} to get that some universal $c_0>0$:
$$\E[e^{\lambda Z_i}|\mathcal{F}_{i-1}]\leq \exp\{2(c_0c_1c_2)^2\lambda^2\}$$
for $|\lambda|=1/(c_0c_1c_2)$. Hence, by \Cref{lem: law of total exp in stoch proceses}, a Chernoff-type bound yields:
$$
\P\left(\frac{1}{T} \sum_{t=1}^T Z_t>\epsilon\right)=\P\left(e^{\lambda\sum_{t=1}^T Z_t}>e^{\lambda T \epsilon}\right) \leq e^{-\lambda T \epsilon} \mathbb{E}\left[e^{\lambda \sum_t Z_t}\right]
\leq \exp\{-\lambda T\epsilon + 2T(c_0c_1c_2)^2\lambda^2\}.
$$
Hence, take $\lambda = 1/(c_0c_1c_2)$ we get
$$
\P\left(\frac{1}{T} \sum_{t=1}^T Z_t>\epsilon\right)
\leq \exp\{-\lambda T\epsilon + 2T(c_0c_1c_2)^2\lambda^2\} = \exp\{T\cdot (2-(c_0c_1c_2)^{-1}\epsilon)\}.
$$
Repeating the arguments for $\{-Z_i\}$, we get the desired concentration.
\end{proof}
\paragraph{An inequality concerning the partial sums of $\log\frac{n-i+1}{n}$'s.}
We adapt two inequalities from prior literature that will be useful. Their proofs involve converting sums into respective integrals.
\begin{lemma}[Adapted from the proof of Lemma 12.1 in \cite{hanin2021non}]\label{lem: partial sums log calculus}
 Fix positive integers $q, m$ satisfying $4 \leq m \leq q$. If $q>m$, then
$$
\sum_{j=1}^{m} \log \left(\frac{1}{1-j / q}\right) \geq \frac{m^2}{2 q}.
$$
and (even for $q=m$)
$$
\sum_{j=1}^m \log \left(\frac{1}{1+j / q}\right) \leq-\frac{2m^2}{3 q}.
$$
\end{lemma}
\section{Remainder of the proofs}\label{sec: remainder}
\subsection{Remaining proofs in Step 2}\label{sec: reduction proof} 
{
The key estimate result \eqref{eqn: small_ball_lya}, is follows immediately once we show that (recall \eqref{eq:wedge-norm-def} and \eqref{eq:XV-def})
    \begin{equation}\label{eqn: equiv format smallball}
        \P\left(-\log\frac{\|X_{N, n}U_0\|_{(k)}}{\|X_{N, n}\|_{(k)}} \geq ck\log (en/k)\right)\leq c_1(en/k)^{-c_2ck}
    \end{equation}
    for $c>c_3$.
}
We prove this estimate by first obtaining an upper bounding the inverse moment of the determinant of $T = \sqrt{n}\one_{[k]}WU$ (and thus also a small-ball estimate this determinant).
\begin{lemma}\label{lem:sq_det_small}
    For all $C_1\leq 1/10$, there exists a constant $C$ with the following property. For all $n\geq k$ and $U\in\fr_{n, k}$, $T = MU\in\R^{k\times k}$ where  $M\sim\mu^{\otimes k\times n}$ satisfies:  $$\mathbb{E}\left[\left(\left|\operatorname{det}\left(T\right)\right|^2 / k!\right)^{-C_1}\right]\leq \prod_{t=1}^k e^{C_2/t} \leq e^{C_2(1+\log k)},\qquad C_2= C K_\infty^{4C_1}M_4,$$
    where $K_\infty$ is the upper bound for the density of $\mu$ in \cref{cond:main_mu}.
\end{lemma}

\begin{proof}
    For matrix $S$ and vector $v$ we define $\operatorname{dist}(v, S) \triangleq  \inf_{w\in\operatorname{span}(S)}\|v-w\|$. 
    Note that 
    $$\log|\det(T)|^2/k!= \sum_{i=1}^k \log\|\operatorname{dist}(T_i, T_{<i})\|^2-\log (k-i+1) \triangleq \sum_{i=1}^k\log (C^{(i)}/i),
    $$
 where for each fixed $i$ we've set
 $$C^{(i)} \triangleq \|\operatorname{dist}(T_j, T_{<j})\|^2,\qquad j = k-i+1.$$
Note that the rows $T_1,\ldots, T_k$ of $T$ are projections of the rows of $M_1,\ldots, M_k$ of $M$ onto the column space of $U$. Hence, we may also write 
\begin{equation}\label{eq:C-rep}
    C^{(i)}=\norm{\mathrm{dist}(M_j, T_{<j})}^2
\end{equation}
Note that $M_j$ is independent of $T_{<j}$. Hence, by \Cref{lem: law of total exp in stoch proceses}, the conclusion of \Cref{lem:sq_det_small} will follow once we show that there exists $C_2>0$ such that for all $t\geq 1$ (and frame $\Theta_t$):
 \begin{equation*}
     \E\left[(C^{(t)}/t)^{-C_1}\right]\leq 1+C_2/t\leq \exp\{C_2/t\}.
 \end{equation*}
To obtain this estimate, we will show that 
\begin{align}
    \label{eq:sq_det_small_goal-1}
    \E\left[(C^{(t)}/t)^{-C_1} {\bf 1}_{\{(C^{(t)} / t)^{-1} > (CK_\infty)^{4}\}}\right]&= O(1/t)\\
   \label{eq:sq_det_small_goal-2} \E\left[(C^{(t)}/t)^{-C_1} {\bf 1}_{\{(C^{(t)} / t)^{-1} <(CK_\infty)^{4}\}}\right]&= 1+O(1/t),
\end{align}
where $C$ is some universal constant $C$ such that $CK_\infty\geq 2$. To obtain \eqref{eq:sq_det_small_goal-1} and \eqref{eq:sq_det_small_goal-2}, we return to \eqref{eq:C-rep} and denote by $\Theta_{i}\in \fr_{n,i}$ a frame consisting an orthonormal basis for the orthogonal complement to $T_{<j}$. We then have
\begin{equation}\label{eq:C-iid-rep}
C^{(i)}=_d\norm{\Theta_i^T u}^2,\qquad u = M_i\text{ is the }i-\text{th row of }M,    
\end{equation}
where $\Theta_i$ is independent of $u$. Since $u\sim \mu^{\otimes n}$, we have has 
\[
\E[C^{(i)}] = i.
\]
Moreover, by \Cref{lem:small_ball_proj_applied},
 $$
 \P\left(\left(C^{(t)}/t\right)^{-C_1}\geq s\right)\leq (C K_\infty\cdot  s^{-1/2C_1})^{t}
 $$
for some universal constant $C$, which we assume is sufficiently large that $CK_\infty\geq 2$. In particular, we have that
\[
    s \geq (CK_\infty)^{4C_1}\quad \Longrightarrow\quad (C K_\infty\cdot  s^{-1/2C_1})^{t}\leq s^{-t/4C_1}
\]
 and hence that
\begin{align*}
\E\left[(C^{(t)}/t)^{-C_1} {\bf 1}_{\{(C^{(t)} / t)^{-1}> (CK_\infty)^{4}\}}\right]& \leq \int_{(CK_\infty)^{4C_1}}^\infty \mathbb P\lr{(C^{(t)}/t)^{-C_1}>s}ds\\
&\leq \int_1^\infty s^{-t/4C_1}ds=\frac{4C_1}{t-4C_1} = O(1/t).    
\end{align*}
This confirms \eqref{eq:sq_det_small_goal-1}. Next, to verify \eqref{eq:sq_det_small_goal-2}, note that since $x^{-C_1}$ is a convex function, there exists a finite $D(K_\infty)$ such that for all $x\geq (CK_\infty)^{-4}$ it is bounded above by it's second order Taylor expansion around $x=1$:
$$
x^{-C_1}\leq 1 + C_1(1-x) + D(1-x)^2.
$$
The right hand side is always positive and hence
\begin{equation}\label{eq:convex-majorant}
    \E\left[\one_{C^{(t)}/t\geq (CK_\infty)^{-4}}(C^{(t)}/t)^{-C_1}\right]\leq \E\left[1 - C_1(1-C^{(t)}/t) + D(1-C^{(t)}/t)^2\right]=1+\frac{D}{t^2}\E\left[(C^{(t)}-t)^2\right].
\end{equation}
To deduce \eqref{eq:sq_det_small_goal-2} we must therefore show that the expression on the right is $1+O(1/t)$. For this observe, recall from \eqref{eq:C-iid-rep} that
\[
    C^{(t)} = \|\Theta_t^T u\|^2.
\]
Write
\[
\mathcal A(\mu):=\E\left[(C^{(t)})^2\right],
\]
where we emphasize the dependence on the distribution $\mu$. Note that $(C^{(t)})^2$ is a degree four polynomial in the $u_i$'s and that 
\[
\mathcal A(\mu) = \sum_{i=1}^n \sum_{s=1}^t \Theta_{si}^2 \E[u_i^4]  + \text{terms that depend only on the first two moments of }\mu.
\]
Since the first two moments of $\mu$ are the same as those of a standard Gaussian we therefore find
\[
\mathcal A(\mu) = (M_4 - 3)t + \mathcal A(\mathcal N(0,1)).
\]
Moreover, writing $\chi_t^2$ for a chi-squared distribution with $t$ degrees of freedom we find
\[
\mathcal A(\mathcal N(0,1)) = \E[\lr{\chi_t^2}^2] = t(t+2).
\]
Hence, 
\[
\mathcal A(\mu) = t^2 + (M_4-1)t
\]
and so 
\[
\E{(C^{(t)}-t)^2} = M_4t.
\]
When combined with \eqref{eq:convex-majorant} this yields 
Therefore $\E\left[(C^{(t)}-t)^2\right]\leq (M_4+5)t$, and 
$$\E\left[(C^{(t)}/t)^{-C_1}\right]\leq 1+\frac{1+DM_4}{t}\leq \exp\left\{\frac{1+DM_4}{t}\right\}$$
where $D = 100(CK_\infty)^{4C_1}$ suffices. This verifies \eqref{eq:sq_det_small_goal-2} and completes the proof.
\end{proof}
\paragraph{Completion of Proof of \Cref{prop:sup_eq_any}.}
Given the above tools, we are now in place to show \Cref{prop:sup_eq_any}, for which we only needed to establish \eqref{eqn: equiv format smallball}. To do this, note that via \Cref{lem: sub mult k norm}:
$$
\|X_{N, n}\|_{(k)}\leq \|W_1\|_{(k)}\|W_N\cdots W_3W_2\|_{(k)}
$$
and that there exists $\Theta_0, L \in\fr_{n, k}$ via the SVD of $W_NW_{N-1}\dots W_2$ such that 
$$
\Theta_0^T\cdot W_{N}\cdots W_3W_2 = \operatorname{diag}(\{s_i(W_{N}\cdots W_3W_2)\}_{i=1}^k)\cdot L^T
$$
and hence the determinant is
$$\|X_{N, n}U_0\|_{(k)}\geq_{\eqref{eq:wedge-norm-def}}|\det(\Theta_0^TX_{N, n}U_0)| = |\operatorname{det}\left(L^T W_1 U_0\right)|\cdot \|W_N\cdots W_2\|_{(k)}\geq |\operatorname{det}(L^T W_1 U_0)|\frac{\|X_{N, n}\|_{(k)}}{\|W_1\|_{(k)}}$$ 
Thus,
$$0\leq -\log\frac{\|X_{N, n}U_0\|_{(k)}}{\|X_{N, n}\|_{(k)}}\leq \log\frac{\|W_1\|_{(k)}}{\left|\operatorname{det}\left(L^T W_1 U_0\right)\right|}=\log\frac{\|M\|_{(k)}}{\left|\operatorname{det}\left(L^T MU_0\right)\right|}$$ 
where $M=\sqrt{n}\cdot W_1\sim\mu^{\otimes n\times n}$ is the un-normalized random matrix. To complete the derivation of \eqref{eqn: equiv format smallball}, it now suffices to show that there exists constants $c_1, c_2, c_3>0$ such that for $G\sim\operatorname{unif}(\fr_{n, k})$ (see \Cref{lem: random wedge product} for the exact definition):
$$
\P\left(\frac{\|M\|_{(k)}^2}{\|MG\|_{(k)}^2}\geq (en/k)^{ck} \right), \P\left(\frac{\|MG\|_{(k)}^2}{k!}\geq(en/k)^{ck}  \right), \P\left(\frac{k!}{|\det(L^TMU_0)|^2}\geq(en/k)^{ck}  \right)
$$
are all at most 
$$c_2(en/k)^{-c_1ck}$$
whenever $c>c_3$. We bound these probabilities separately below:
\begin{enumerate}
    \item We consider, for any full-rank $M$, with randomness over $G$:
    $$\P\left(\frac{\|M\|_{(k)}^2}{\|MG\|_{(k)}^2}\geq (en/k)^{ck}  \right) = \mathbb{P}\left(\frac{\|MG\|_{(k)}^2}{\|M\|_{(k)}^2} \leq \left(\frac{k}{en}\right)^{ck}\right) 
    $$
    This quantity is bounded directly by \Cref{lem: random wedge product}  which states that the above objective is at most:
    $$\mathbb{P}\left(\left(\frac{\|MG\|_{(k)}}{\|M\|_{(k)}}\right)^{\frac{1}{2 k}} \leq \varepsilon \sqrt{\frac{k}{n}}\right) \leq(c \varepsilon)^{\frac{k}{2}}$$
    for any $\eps$ and some universal constant $c$. This means that
    $$\P\left(\frac{\|M\|_{(k)}^2}{\|MG\|_{(k)}^2}\geq (en/k)^{ck} \right) \leq\left(\frac{en}{k}\right)^{c_0k-ck/4}$$
    for a universal $c_0$.
    \item We show that one has, for any $G$:
    $$\P\left(\frac{\|MG\|_{(k)}^2}{k!}\geq(en/k)^{ck}   \right)\leq  \left(\frac{en}{k}\right)^{k-ck}
    $$
    This is because by \eqref{eqn: exp norm}:
    $$\E\left[\frac{\|MG\|_{(k)}^2}{k!} \right]=\frac{n(n-1)\cdots(n-k+1)}{k!}\leq \frac{n^k}{k!}\leq (en/k)^k$$
    \item For $U_0$ being the truncated identity, by \Cref{lem:sq_det_small}, there exists $C_1, C_2$ such that:
    $$\P\left(\frac{k!}{|\det(L^TMU_0)|^2}\geq(en/k)^{ck} \right)\leq  (en/k)^{-C_1ck}(ek)^{C_2}\leq (en/k)^{C_2(k+1)-C_1ck}$$
    holds for any $L\in\fr_{n, k}$ directly, since $(en/k)^{k+1}\geq e^{k+1}>ek$.
\end{enumerate}
Combining these three points along with a union bound concludes our proof of \eqref{eqn: equiv format smallball} and \Cref{prop:sup_eq_any}.\hfill $\square$

\subsubsection{A different proof without restricting $U_0$}\label{sec: any U0}
Of perhaps separate interest, we show a similar result to \Cref{prop:sup_eq_any} without restricting on $U_0$ to be the truncated identity. Specifically, we show that:
\begin{proposition}
    Assuming \Cref{cond:main_mu}, there exist constants $c_1, c_2, c_3>0$ depending only on $K_\infty, M_4$, such that for any $n\geq k$ and any $U_0\in\fr_{n, k}$:
    \begin{equation}\label{eqn: small_ball_lya_any}
    \mathbb{P}\left(\frac{1}{n}\left|\sum_{j=1}^k \lambda_i-\frac{1}{N}\log \left\|X_{N, n}U_0\right\|_{(k)}\right|>s\right) \leq c_1 \exp \{-c_2 n N s/k\}.
    \end{equation}
    for all $ s>c_3\frac{k\log(en)}{nN}$. 
\end{proposition}
Following the exact same recipe for the proof in \Cref{prop:sup_eq_any}, the only distinction lies in the following lemma, which we find interesting in its own rights.
\begin{lemma}\label{lem:small_ball_one_mat}
    For any fixed frames $U, V\in\fr_{n, k}$ and $M\in\mu^{\otimes n\times n}$ where $\mu$ satisfies \Cref{cond:main_mu}, one has that for all $t>0$:
$$\mathbb{P}\left(\left|\operatorname{det}\left(U^T M V\right)\right|^{1 / k} \leq n^{-(t+1)} k^{-1/2}\right) \leq 2 \sqrt{2} K_\infty n^{-t} k.$$
\end{lemma}
\begin{proof}
First, we look at what linear transformations we can do to $U, V$ while preserving $\det(U^TMV)$:\begin{itemize}
    \item Adding a constant multiple of one column to another column. This, by definition, does not change the determinant.
    \item Row exchanges. This changes the determinant but preserves the \emph{law} of $\det(U^TMV)$ as the law of $M$ is invariant with respect to row and column exchanges.
\end{itemize} 
Note that for any frame, these two operations allow us to turn them into the form:
$$
U^T, V^T\to \widetilde U^T, \widetilde V^T = \begin{bmatrix}
a_1 & 0 & 0 & \cdots & 0 & \cdots\\
0 & a_2 & 0 & \cdots & 0 & \cdots\\
\vdots & \vdots & \vdots & \ddots & \vdots \\
0 & 0 & 0 & \cdots & a_k & \cdots
\end{bmatrix}, 
\begin{bmatrix}
b_1 & 0 & 0 & \cdots & 0 & \cdots\\
0 & b_2 & 0 & \cdots & 0 & \cdots\\
\vdots & \vdots & \vdots & \ddots & \vdots \\
0 & 0 & 0 & \cdots & b_k & \cdots
\end{bmatrix}
$$
where $|a_i|, |b_j|\geq\frac{1}{\sqrt{n}}$, by following the algorithm below.
\begin{mdframed}[innerrightmargin=15pt]
 \noindent\textbf{Algorithm: Diagonalizing frame $U\in\fr_{n, k}$ column by column.}
 
    \vspace{-5pt}
    \begin{enumerate}
        \item Initialize $U = U^{(0)}$. For $t=1,2,\dots, k$, do the following to get $U^{(t)}$ from $U^{(t-1)}$:  
        
        (1) row-exchanging the argmax (of absolute value) entry row of the $t$-th column $(U^{(t-1)})^T_t$ to the $t$-th row; 
        
        (2) use column elimination (adding appropriate scalar multiples of the $t$-th column) to make the $t$-th row all zero except at the $t$-th column.
        \item Output $\tilde U = U^{(k)}$.
    \end{enumerate}
\end{mdframed}

To analyze this procedure, note that under time $t$, the norm of the $t$-th column is at least 1 because (ignoring row exchanges which are irrelevant) it has only been added a linear combination of the first $(t-1)$ columns which are all orthogonal to itself. Hence, the arg-max absolute value is at least $1/\sqrt{n}$ at time $t$. Hence, the result $U^{(k)}$ must have the top $k\times k$ submatrix diagonalized with diagonal entries at least $1/\sqrt{n}$.

Note that we can take any $k\times k$ matrices that have bounded determinant and left (right) multiply to our product. Let two $k\times k$ matrices be
$$L, R = \begin{bmatrix}
a_1^{-1} & 0 & \cdots & 0\\
0 & a_2^{-1}& \cdots & 0 \\
\vdots & \vdots & \ddots &\vdots  \\
0 & 0 & \cdots & a_k^{-1} 
\end{bmatrix}, 
\begin{bmatrix}
b_1^{-1} & 0  & \cdots & 0\\
0 & b_2^{-1}  & \cdots & 0 \\
\vdots &  \vdots & \ddots &\vdots  \\
0 & 0  & \cdots & b_k^{-1} 
\end{bmatrix}, $$
then both $A= \widetilde  U L$ and $ B= \widetilde VR$ share the form of $ \begin{bmatrix}
1 & 0 &\cdots & 0 & \cdots\\
0 & 1 & \cdots & 0 & \cdots\\
\vdots  & \vdots & \ddots & \vdots \\
0 & 0  & \cdots & 1 & \cdots
\end{bmatrix}.$ Note that since $$|\det(L)|, |\det(R)|\leq n^{k/2}$$ we only need to study the small ball probability for 
\begin{equation}\label{eqn: det transformation}
    |\det(U^TMV)|=_d|\det(\tilde U^TM\tilde V)|=
|\det(A^TMB)| \cdot\left( |\det(L)| |\det(R)|\right)^{-1}\geq n^{-k}|\det(A^TMB)|.
\end{equation}

To analyze this determinant, we need the following result. 
\begin{lemma}\label{lem: det X+M}
    Under \Cref{cond:main_mu}, let $M\sim \mu^{\otimes k\times k}$. Fix any $X\in\R^{k\times k}$, one has:
    $$\P\left(|\det(X+M)|<\left(n^{-t}\sqrt{k^{-1}}\right)^k\right)\leq 2\sqrt{2}K_\infty n^{-t}k$$
\end{lemma}
\begin{proof}
    Note this simple observation (Lemma 5.1, \cite{tao2010random}): let $N$ be any matrix and then the $i$-th row of $J\triangleq (N^{-1})^T$ satisfies:
$$\|J_i\|^{-1}=\|(N^{-1})_i\|^{-1} = \min_{c_j\in\R}\| N_i-\sum_{j\neq i} c_jN_j\|$$and this is simply because $\langle J_i, N_i-\sum_{j\neq i} c_jN_j\rangle = 1$ always for any $\{c\}$. Hence, $$\sigma_{\min}(N)=\sigma^{-1}_{\max}(J)\geq \|J\|_F^{-1} \geq \sqrt{k^{-1}}\min_i \min_{c_j\in\R^{k-1}, j\neq i}\left\| N_i-\sum_{j\neq i} c_jN_j\right\|.$$
To use a union bound on the $k$ rows of $N = X+M$, we only need to show that for any fixed $i$,
$$\P\left(\min_{c_{-i}\in\R^{k-1}}\left\| N_i-\sum_{j\neq i} c_jN_j\right\|\leq n^{-t}\right)\leq 2\sqrt{2}Kn^{-t}.$$
This is because, fix $M_{-i}$ and only consider $M_i\sim\mu^{\otimes k}$. Let the unit vector of null space $\{N_{-i}\}^{\perp}$ (which is independent with $M_i$) be $w_i$ then:
$$\min_{c_{-i}\in\R^{k-1}}\left\| N_i-\sum_{j\neq i} c_jN_j\right\| = |\langle M_i, w_i\rangle+\langle X_i, w_i\rangle|$$ where $X_i, w_i$ are $\sigma(M_{-i})$-measurable. The probability that this is small is directly concluded \Cref{lem:small_ball_proj_rv}.
\end{proof}
To complete the proof, let us write the product $A^TMB$ as a linear combination of $$A^TMB = \sum_{1\leq i, j\leq n}M_{ij}A_i^TB_j$$ where $A_i, B_j\in\R^{1\times k}$.
Note that $A_i^TB_j = E^{(i,j)}$ for $1\leq i, j\leq k$ where $E^{(i,j)}$ denotes the rank-1 matrix with only the $(i, j)$th entry being 1 and else 0. Hence, conditioned on the irrelevant entries $\sigma(\{M_{i j}: i\lor j\geq k\})$ (treating them as constant) we get
$$A^TMB \triangleq M_{[k], [k]}+X,\quad\quad X\perp\!\!\!\perp M_{[k], [k]}$$
applying \Cref{lem: det X+M} and
combining with \eqref{eqn: det transformation} we are done.
\end{proof}
\subsection{Remaining proofs in Step 3: Derivation of \Cref{prop: ptwise any}}\label{sec: ptwise proof}
Our goal in this section is to derive \Cref{prop: ptwise any}. As mentioned in  \Cref{sec: proof short} the main idea is to express the norm $\log\|X_{N, n}U_0\|_{(k)}$ as an average.
To do this, we repeatedly use the SVD yields to obtain an alternative representation of $\|X_{N, n}U\|_{(k)}$ as follows. First, let $U_1 = U_0$ and $U_{t+1}$ be defined via (for $t=1,2,\dots, N-1$) the singular value decomposition of $W_tU_t$ as: $$W_tU_t = U_{t+1}\operatorname{diag}(\{s_i(W_tU_t)\}_{i=1}^k)O_{(t)}$$
for $O_{(t)}\in\fr_{k, k}$ and $U_{t+1}\in\fr_{n, k}$. Then $\|X_{N, n}U_0\|_{(k)}$ can be also written as (recall \eqref{eq:wedge-norm-def}):
\begin{align*}
    \|X_{N, n}U_0\|_{(k)}
    &=\sup_{V}\det(V^TX_{N, n}U_0) \notag\\
    &= \sup_{V}\det(V^TW_{N}U_{N})\cdot\prod_{i=1}^{N-1}\det\left[\operatorname{diag}(\{s_i(W_iU_i)\}_{i=1}^k)\right]\notag
    \\&=\prod_{i=1}^N\|W_iU_i\|_{(k)}.
\end{align*}
Since $U_t\in\sigma(W_1^{t-1})$ is independent with $W_t$ for all $t$, we only need to study the objective in \eqref{eqn: point norm prod decomp}:
\begin{equation*}
\frac{1}{N}\log \|X_{N, n}U\|_{(k)}^2 = \frac{1}{N}\sum_{i=1}^N \log \|W_iU_i\|_{(k)}^2.
\end{equation*}

Let us define the $n\times k$ product matrix $T^{(i)} = \sqrt{n}W_iU_i$ where each rows are independent (conditioned on $U_i$) from the following law
$$ T_j\sim_{i.i.d.} U_i^Tw\mid_{w\sim\mu^{\otimes n}}$$
with $\E[(T_j)_i] = 0, \E[(T_j)_i^2] = 1, \E[((T_j)_{i_1}(T_j)_{i_2}] = 0$ for all $i, j, i_1\neq i_2$. Thus, for any $k\times k$ submatrix $T_I$ (indexed by a row subset $I\in{\binom{[n]}{k}}$ of size $k$), the expected determinant squared (by independence of rows) is exactly
$$\E[|\det(T_I)|^2] = \sum_{\substack{i_1, i_2,\dots, i_k\\ j_1, j_2, \dots, j_k}}\prod_{s=1}^k\E\left[ T_{I_si_s}T_{I_sj_s}\right] = \sum_{\substack{i_1, i_2,\dots, i_k\\ j_1, j_2, \dots, j_k}}\prod_{s=1}^k\one\left[ i_s = j_s\right] = k!.$$
This gives
\begin{equation*}
    \E\left[\|W_iU_i\|_{(k)}^2\right] =\E[\det\left((W_iU_i)^TW_iU_i\right)]= \frac{1}{n^k}\sum_{I\in {\binom{[n]}{k}}} \E|\det(T_I)|^2 = \prod_{j=1}^k\frac{n-j+1}{n}.
\end{equation*}
To complete the proof of \Cref{prop: ptwise any} note that 
\begin{align}
    &\P\left(\frac{1}{n}\left|\log\|W_iU_i\|_{(k)}^2-\sum_{j=1}^k\log \frac{n-j+1}{n}\right|>s\right)
    \leq \exp\{-sn\} + \P\left(\frac{\|W_iU_i\|_{(k)}^2}{\frac{n!}{(n-k)!\cdot n^k}}<\exp\{-sn\}\right), \label{eqn: up and low}
\end{align}
where the upper tail $\P\left(\frac{\|W_iU_i\|_{(k)}^2}{\frac{n!}{(n-k)!\cdot n^k}}>\exp\{-sn\}\right)\leq e^{-sn}$ follows from Markov's inequality. Observe that\Cref{lem:sq_det_small} gives the following inverse moment bound:
$$\mathbb{E}\left[\left(\left|\operatorname{det}\left(T_I\right)\right|^2 / k!\right)^{-C_1}\right] \leq e^{C_2(1+\log k)}$$
for some constant $C_1, C_2$ depending only on $M_4, K_\infty$.
Therefore, one can apply Jensen's inequality on $f(x)=x^{-C_1}$ to get:
\begin{align}
    \E\left[\left(\|WU\|_{(k)}^2/\frac{n!}{(n-k)!\cdot n^k}\right)^{-C_1}\right] 
    &= \E\left[\left(\frac{1}{{\binom{n}{ k}}}\sum_{|I|=k} |\det(T_{I})|^2/k!\right)^{-C_1}\right]\notag
    \\
    _{\text{Jensen's inequality}}&\leq\E\left[\frac{\sum_{|I|=k}(|\det(T_{I})|^2/k!)^{-C_1} }{{\binom{n}{ k}}}\right]\leq (ek)^{C_2}.\notag
\end{align}
This gives the following bound on the lower tail:
\begin{equation}\label{eqn: small ball nk mat}
    \P\left(\|WU\|_{(k)}^2/\frac{n!}{(n-k)!\cdot n^k}<e^{-sn}\right)\leq\exp\{-C_1sn+ C_2\log(ek))\}.
\end{equation}
To conclude the proof of \Cref{prop: ptwise any}, we combine \eqref{eqn: point norm prod decomp}, \eqref{eqn: small ball nk mat} and \eqref{eqn: mtg large off} in \Cref{lem: mtg concentrate} with 
$$Z_i = \frac{1}{n}\left(\log\|W_iU_i\|_{(k)}^2-\sum_{i=1}^k\log\frac{n-i+1}{n}\right),$$ 
to see that the conditions for \Cref{lem: mtg concentrate} hold with $c_1\in \Theta(1)$ and $c_2\in \Theta(\frac{\log (ek)}{n})$. We get as a result that for some constants $c_0, c_3$ depending only on $K_\infty, M_4$ and for all $s\geq c_3 n^{-1}\log (ek)$:
$$
\P\left(\frac{1}{nN}\left|\log \left\|X_{N, n} U\right\|_{(k)}^2-N\sum_{j=1}^k\log\frac{n-j+1}{n}\right|>s\right)\leq 2\exp\{-c_0 nNs/\log(ek)\}.
$$
This concludes the proof of \eqref{eqn: concentrate single matt}. \hfill $\square$

\subsection{Completion of Proof of \Cref{thm:sing_val_approx_unif}}\label{sec: main proofs end}
We are now in a position to complete the proof of \Cref{thm:sing_val_approx_unif}. For this, note that \eqref{thm: main technical concentration} follows immediately when combining \eqref{eqn: small_ball_lya}, \eqref{eqn: concentrate single matt}, and a union bound. Given this, we can check \Cref{thm:sing_val_approx_unif} via a union bound as follows. 

By twisting constants (multiplying by universal constants) in \eqref{eqn: thm 1 main}, we may assume that $\eps > 5/n$ and that $\eps<0.01$. Note that under the given conditions of $N\geq c_1\eps^{-2}, n\geq c_2\eps^{-2}\log\eps^{-1}$, one has that (via \eqref{thm: main technical concentration}):
    $$
    \mathbb{P}\left(\left|\frac{1}{n} \sum_{i=m}^k\left(\lambda_i-\frac{1}{2} \log \frac{n-i+1}{n}\right)\right| \geq \eps^2\right) \leq C_2 \exp \left\{-C_3 n N \eps^2 / \log(ek)\right\}\leq n^{-4}
    $$
    if we pick large enough constants $c_1, c_2$. As a result, a union bound may be applied such that
    $$
    \mathbb{P}\left(\exists m<k, \quad \left|\frac{1}{n} \sum_{i=m}^k\left(\lambda_i-\frac{1}{2} \log \frac{n-i+1}{n}\right)\right| \geq \eps^2\right) \leq C_2 \exp \left\{-\tilde C_3 nN \eps^2/\log n \right\}.
    $$ for some $\tilde C_3>0$.
    In fact, we will show that so long as:
    \begin{equation}\label{eqn: unif condition}
        \left|\frac{1}{n} \sum_{i=m}^k\left(\lambda_i-\frac{1}{2} \log \frac{n-i+1}{n}\right)\right| \leq \eps^2
    \end{equation}
    holds for all $1\leq m < k\leq n$, one has that for all $1\geq t\geq 0$:
    \begin{equation}\label{eqn: sup distance}
        |F(t)| \triangleq \left|\frac{1}{n} \#\left\{1 \leq i \leq n \mid s_i^{2 / N}\left(X_{N, n}\right) \leq t\right\}-t\right| \leq 5\varepsilon
    \end{equation}
    which, if true, concludes the proof of \Cref{thm:sing_val_approx_unif} (by, again, twisting constants). First let us check two basic inequalities following from \Cref{lem: partial sums log calculus} directly.
    In particular, it follows that for any $n\geq q\geq m\geq 5$ one has:
      \begin{equation}\label{eqn: lemma 12.1 upp}
          m\log (q / n)- \sum_{j=n-q+1}^{n-q+m} \log\frac{n-j+1}{n} = \sum_{j=1}^{m-1}\log\frac{1}{1-j/q} \geq \frac{(m-1)^2}{2  q}.
      \end{equation}
    Furthermore for any $n\geq q\geq m\geq 5, n-q-m \geq 0$, we also have
    \begin{equation}\label{eqn: lemma 12.1 low}
        m \log (q / n)- \sum_{j=n-q-m+1}^{n-q} \log\frac{n-j+1}{n} = \sum_{j=1}^m \log \left(\frac{1}{1+j / q}\right)\leq-\frac{2m^2}{3 q}.
    \end{equation}
    To show \eqref{eqn: sup distance}, it suffices to check that for $t=1$ and $t=s/n$ where $1\leq s\leq n-1$ is an integer. The reason is that $\frac{1}{n} \#\left\{1 \leq i \leq n \mid s_i^{2 / N}\left(X_{N, n}\right) \leq t\right\}$ is non-decreasing. Hence if $nt\in(s, s+1)$ where $0\leq s\leq n-1$ then
    $$|F(t)|\leq \frac{1}{n} + |F(s/n)|\lor |F((s+1)/n)|\leq \eps + |F(s/n)|\lor |F((s+1)/n)|$$
    and $0\geq F(t)\geq F(1)$ if $t\geq 1$. The case for $t\leq 0$ is trivial. Suppose $\alpha\leq n\eps <  \alpha+1$ where $\alpha\geq 5$ is an integer. We will show that $|F(t)|\leq 4\eps$ for $t=s/n, s=1,2,\dots, n$.

    \paragraph{The case of $t=1$}
    Suppose $\lambda_1\geq \lambda_2\dots\geq \lambda_r \geq 0>\lambda_{r+1}$. Then $|F(1)| = r/n$ and \eqref{eqn: unif condition} reads:
    $$
    n\eps^2\geq \sum_{i=1}^r\left(\lambda_i-\frac{1}{2} \log \frac{n-i+1}{n}\right) \geq -\frac{1}{2}\sum_{i=1}^r\log\frac{n-i+1}{n}.
    $$
    If $r\leq 5\leq 3\alpha$ then \eqref{eqn: sup distance} is already closed. Otherwise, by \eqref{eqn: lemma 12.1 upp} with $q=n, m=r$ one gets
    $$
    n\eps^2\geq-\frac{1}{2}\sum_{i=1}^r\log\frac{n-i+1}{n}\geq \frac{(r-1)^2}{2n} > \frac{r^2}{4n}.
    $$
    which implies that $r/n<2\eps$.
    \paragraph{The case of $t=s/n$, where $s\in [1, n-1]$ is an integer}
    
    Suppose $$\lambda_1\geq \lambda_2\dots\geq \lambda_r > \frac{1}{2}\log(s/n)\geq \lambda_{r+1}\geq \dots\geq \lambda_n.$$ Then $|F(t)| = |r+s-n|/n$. Again, if $|r+s-n|\leq 5$, then we are already done. 
    \begin{itemize}
        \item If $r+s-n\geq 5$, then our condition \eqref{eqn: unif condition}, combined with \eqref{eqn: lemma 12.1 upp} with $q=s, m = r+s-n$ reads:
        $$
    2n\eps^2\geq \sum_{i=n-s+1}^{r}\left(2\lambda_i- \log \frac{n-i+1}{n}\right) \geq m\log(s/n)-\sum_{j=n-s+1}^{r} \log\frac{n-j+1}{n}\geq \frac{(m-1)^2}{2s}
        $$
        or
    $(n|F(t)|-1)^2\leq 4ns\eps^2\leq 4n^2\eps^2$. This implies that $|F(t)|\leq 1/n+2\eps<3\eps$.
    \item If $n-s-r\geq 5$. If $n-r<\alpha$ then we are already done. Assume otherwise, if $n-r\leq 2s$, then \eqref{eqn: lemma 12.1 low} with $q=s, m = n-r-s$ reads
    $$
    2n\eps^2\geq \sum_{i=r+1}^{n-s}\left(-2\lambda_i + \log \frac{n-i+1}{n}\right) \geq -m\log(s/n)+\sum_{j=r+1}^{n-s} \log\frac{n-j+1}{n}\geq \frac{2m^2}{3s}
    $$
    which implies, as before, $|F(t)|\leq 2\eps<3\eps.$ If $n-r>2s$, then \eqref{eqn: lemma 12.1 low} with $q=s=m$ reads
    $$
    2n\eps^2\geq \sum_{i=n-2s+1}^{n-s}\left(-2\lambda_i + \log \frac{n-i+1}{n}\right) \geq -s\log(s/n)+\sum_{j=n-2s +1}^{n-s} \log\frac{n-j+1}{n}\geq \frac{2s}{3}
    $$
    which implies $s\leq 3n\eps^2<\alpha$ and $t<\eps$. In this case, note that:
    $$|F(t)| = \frac{n-r}{n}-t\leq \left|F\left(\frac{\alpha+1}{n}\right)\right|+\frac{\alpha+1-s}{n}\leq 3\eps+\frac{\alpha}{n}$$
    from our previous discussion. This concludes our claim.
    \end{itemize}
 Our proof is concluded. \hfill $\square$

\bibliography{ref} 

\end{document}